\documentclass[a4paper,11pt]{amsart}

\usepackage[english]{babel}
\usepackage[a4paper,margin=2.4cm,marginparwidth=2.2cm]{geometry}
\usepackage{amsmath,amsthm,amssymb,mathtools}
\usepackage{thmtools}
\usepackage{xcolor}
\usepackage{paralist}
\usepackage{tikz}

\definecolor{Chocolat}{rgb}{0.36, 0.2, 0.09}
\definecolor{BleuTresFonce}{rgb}{0.215, 0.215, 0.36}
\usepackage[colorlinks,final]{hyperref}
\hypersetup{citecolor=BleuTresFonce, linkcolor=Chocolat, urlcolor=BleuTresFonce}
\usepackage[noabbrev,capitalize]{cleveref}

\newcommand{\RR}{\mathbb{R}}
\newcommand{\ZZ}{\mathbb{Z}}
\newcommand{\QQ}{\mathbb{Q}}
\newcommand{\NN}{\mathbb{N}}
\newcommand{\DD}{\mathbb{D}}

\DeclareMathOperator{\conv}{conv}
\DeclareMathOperator{\aff}{aff}
\DeclareMathOperator{\lin}{lin}          
\DeclareMathOperator{\dir}{dir}          
\DeclareMathOperator{\coker}{coker}
\DeclareMathOperator{\GL}{GL}
\DeclareMathOperator{\nvol}{nvol}
\DeclareMathOperator{\relint}{relint}
\DeclareMathOperator{\vertt}{vert}
\DeclareMathOperator{\Tight}{Tight}

\numberwithin{equation}{section}
\theoremstyle{plain}
\newtheorem{theorem}[equation]{Theorem}
\newtheorem{proposition}[equation]{Proposition}
\newtheorem{corollary}[equation]{Corollary}
\newtheorem{lemma}[equation]{Lemma}

\theoremstyle{definition}
\newtheorem{definition}[equation]{Definition}

\newtheorem{remark}[equation]{\sc Remark}
\newtheorem{example}[equation]{\sc Example}

\title{Regular dyadic triangulations of delta-matroid polytopes}

\hypersetup{
  pdftitle={Regular dyadic triangulations of delta-matroid polytopes},
  pdfauthor={Mathieu Vall\'ee},
  pdfsubject={Regular triangulations of delta-matroid polytopes and type B generalized permutohedra},
  pdfkeywords={delta-matroid, type B generalized permutohedron, regular triangulation,
    dyadic number, total dual dyadicness, lattice polytope}
}

\date{September 2026}

\author{Mathieu Vall\'ee}
\address{Universit\'e libre de Bruxelles, CP212, Boulevard du Triomphe, 1050 Brussels, Belgium}
\email{mathieu.vallee@protonmail.com}



\thanks{The author was supported by the Fonds de la Recherche Scientifique-FNRS under Grant
no.~T003325F}

\begin{document}

\begin{abstract}
Backman and Liu proved that every integral generalized permutohedron of
type~$A$, and in particular every matroid base polytope, admits a regular
unimodular triangulation. The analogous statement fails in type~$B$: the delta-matroid simplex
\[
    \conv\{\mathbf{0},\ e_1+e_2,\ e_1+e_3,\ e_2+e_3\}
\]
has normalized volume~$2$ and no lattice points other than its vertices, so it has no unimodular
triangulation. We show moreover that, up to the natural symmetries of the $0/1$ cube and deletion
of constant coordinates, it is the unique non-unimodular delta-matroid polytope that is a simplex.

We prove instead that every delta-matroid polytope admits a regular \emph{dyadic triangulation},
meaning a lattice triangulation whose maximal simplices have normalized volumes that are powers
of two. More generally, every integral type~$B$ generalized permutohedron admits such a
triangulation. The main lattice-theoretic ingredient is that the type~$B$ root configuration forms
a totally dyadic system, a $2$-local analogue of total unimodularity.

As a consequence, these polytopes satisfy a dyadic version of the integer decomposition property. In each dimension the corresponding exponent can be chosen uniformly, even though ordinary integer decomposition can fail for delta-matroid polytopes.
\end{abstract}

\maketitle


\section*{Introduction}
A regular unimodular triangulation gives a lattice polytope strong arithmetic and algebraic
properties, including the integer decomposition property and a squarefree initial degeneration of
its toric ideal. Backman and Liu proved that every integral generalized permutohedron of type~$A$
admits such a triangulation~\cite{Backman_Liu_2025}. Their proof combines deletion--contraction,
total unimodularity of the type~$A$ root configuration, and genericity with respect to the
resonance arrangement. We ask what survives for type~$B$, whose $0/1$ generalized
permutohedra are precisely the delta-matroid polytopes.

Throughout, lattice notions and normalized volume refer to $\ZZ^n$. Since $B_n$ and $C_n$ have
the same root directions and Coxeter fan, our results may also be phrased in type~$C$ provided this
lattice is retained.\footnote{For the simplex below, the root lattice
$Q(C_3)=\{z\in\ZZ^3:\sum_i z_i\in2\ZZ\}$ makes it unimodular, whereas the coweight lattice
$P^\vee(C_3)=\ZZ^3\cup((\tfrac12,\tfrac12,\tfrac12)+\ZZ^3)$ contains the interior point
$(\tfrac12,\tfrac12,\tfrac12)$; stellar subdivision at that point is regular and unimodular.}
We use type~$B$ terminology because its short roots $e_i$ are primitive in $\ZZ^n$.

The unimodular statement already fails for the simplex
$T=\conv\{\mathbf0,e_1+e_2,e_1+e_3,e_2+e_3\}$ in
\Cref{figure:3-simplex_det_2}. It is the polytope of the even delta-matroid
$\{\emptyset,\{1,2\},\{1,3\},\{2,3\}\}$, hence is also of type~$D$. Its normalized volume is
$2$, while its only lattice points are its vertices, so it has no unimodular triangulation.
Moreover, \Cref{proposition:delta_matroid_simplex_volume} shows that, up to cube symmetries and
constant coordinates, this is the unique non-unimodular delta-matroid polytope which is itself a
simplex.

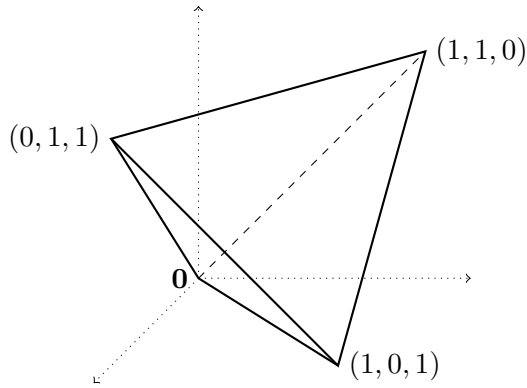
\begin{figure}
\centering
\begin{tikzpicture}[scale=3]
\draw[->,dotted] (0,0,0) -- (1.2,0,0);
\draw[->,dotted] (0,0,0) -- (0,1.2,0);
\draw[->,dotted] (0,0,0) -- (0,0,1.2);
\draw[dashed] (0,0,0) node[left] {$\mathbf{0}$}
    -- (1,1,0) node[right] {$(1,1,0)$};
\draw[thick] (0,0,0)
    -- (0,1,1) node[left] {$(0,1,1)$};
\draw[thick] (0,0,0)
    -- (1,0,1) node[right] {$(1,0,1)$};
\draw[thick] (1,1,0) -- (0,1,1) -- (1,0,1) -- cycle;
\end{tikzpicture}
\caption{The delta-matroid polytope of
$\{\emptyset,\{1,2\},\{1,3\},\{2,3\}\}$: a $0/1$ simplex with type~$B$
(indeed type~$D$) edge directions, normalized volume $2$, and no lattice points besides its
vertices. Hence it admits no unimodular triangulation.%
\label{figure:3-simplex_det_2}}
\end{figure}

This obstruction suggests localizing at $2$. We call a lattice simplex \emph{dyadic} if its
normalized volume is a power of two, equivalently if it becomes unimodular over
$\DD=\ZZ[1/2]$. The type~$B$ root directions form a totally dyadic system: every subconfiguration
generates the saturated lattice of its span over $\DD$ (\Cref{lemma:type_B_TDS}). This is the
$2$-local replacement for total unimodularity, closely related to the dyadic generating sets of
Abdi, Cornu\'ejols, Guenin and Tun\c{c}el
~\cite{Abdi_Cornuejols_Guenin_Tuncel_2024,Abdi_Cornuejols_Guenin_Tuncel_2025}.

\begin{theorem}[= \Cref{theorem:main}]
Every delta-matroid polytope admits a regular dyadic triangulation.
\end{theorem}

The proof follows the Backman--Liu deletion--contraction induction. A generic perturbation for a
symmetric resonance arrangement controls the face-direction spaces; total dyadicness, the lattice
criterion of \Cref{lemma:dyadic_bridge}, and the join lemma then replace the unimodular part of
their argument. Thus the method is recognizably parallel, but its arithmetic content is genuinely
different and the relaxation is sharp already for $T$.

For an arbitrary integral type~$B$ generalized permutohedron, dicing by the integer coordinate
hyperplanes produces translates of delta-matroid polytopes
~\cite[Proposition~2.12]{Eur_Fink_Larson_Spink_2024}; the standard regular dicing construction
~\cite[Section~2.1.3]{Haase_Paffenholz_Piechnik_Santos_2021} then globalizes the triangulations.

\begin{theorem}[= \Cref{corollary:typeB_GP}]
Every integral type~$B$ generalized permutohedron, with integrality taken in $\ZZ^n$,
admits a regular dyadic triangulation. Equivalently at the level of root directions and Coxeter
fans, the same statement holds in type~$C$ with this same lattice convention.
In particular, this holds for every integral type~$D$ generalized permutohedron.
\end{theorem}

Ordinary integer decomposition can also fail for delta-matroid polytopes
~\cite{Morales_2026}. Dyadic triangulations nevertheless give a dimension-uniform saturation.

\begin{theorem}[= \Cref{corollary:dyadic_IDP}]
For every $n$ there exists an integer $m_n\ge 0$ such that, for every integral type~$B_n$
generalized permutohedron $P\subseteq\RR^n$, every $k\ge1$, and every
$x\in kP\cap\ZZ^n$, there are lattice points
$p_1,\ldots,p_{2^{m_n} k}\in P\cap\ZZ^n$ such that
$2^{m_n}x=p_1+\cdots+p_{2^{m_n}k}$.
\end{theorem}

\section{Preliminaries}\label{section:preliminaries}

\subsection{Notation}

We write $[n]=\{1,\ldots,n\}$ and let $e_1,\ldots,e_n$ be the canonical basis of $\RR^n$. For
$X\subseteq[n]$ we write $x_X\coloneqq\sum_{i\in X}x_i$.

For a set $D$ of vectors of $\RR^n$ we write $\lin(D)$ for its linear span and $\ZZ\langle
D\rangle$ for the lattice it generates over $\ZZ$. For a nonempty polytope $P\subseteq\RR^n$ we
write $\vertt(P)$ for its vertex set, $\aff(P)$ for its affine hull and
\[
    \dir(P)\;\coloneqq\;\aff(P)-\aff(P)\;=\;\lin\{p-q\mid p,q\in P\}
\]
for its \emph{direction space}, the linear subspace of which $\aff(P)$ is a translate.

All polytopes in this paper are \emph{lattice} polytopes, i.e.\ $\vertt(P)\subseteq\ZZ^n$. A
subspace $V\subseteq\RR^n$ is \emph{rational} if it is spanned by rational vectors; its
\emph{saturated lattice} is $\Lambda_V\coloneqq V\cap\ZZ^n$, a sublattice of $\ZZ^n$ of rank
$\dim V$ which is saturated, i.e.\ $\ZZ^n/\Lambda_V$ is torsion-free. For a $d$-dimensional lattice
polytope $P$ we write $\nvol(P)$ for its normalized volume, that is, $d!$ times its Euclidean
volume computed relative to the lattice $\dir(P)\cap\ZZ^n$; it is a positive integer. If $P$ is subdivided into lattice polytopes, its normalized
volume is the sum of the normalized volumes of the full-dimensional cells.

\subsection{Delta-matroids and their polytopes}

For $X\subseteq[n]$, let $e_X\coloneqq\sum_{i\in X}e_i$. The \emph{characteristic polytope} of a
collection $\mathcal C\subseteq2^{[n]}$ is
$P_{\mathcal C}\coloneqq\conv\{e_X\mid X\in\mathcal C\}$, with $P_\emptyset=\emptyset$.
Nonempty collections thus correspond bijectively to nonempty $0/1$ polytopes.

For $i\in[n]$, the \emph{contraction} $\mathcal C/i$ consists of the sets $X\setminus\{i\}$ with
$X\in\mathcal C$ and $i\in X$, whereas the \emph{deletion} $\mathcal C\setminus i$ consists of
the $X\in\mathcal C$ with $i\notin X$. Both are regarded as collections on
$[n]\setminus\{i\}$.

\begin{proposition}\label{proposition:polytope_contr_dele}
    Let $\mathcal C$ be a collection of subsets of $[n]$. Then
    \[
        \{0\}\times P_{\mathcal C\setminus 1}=P_{\mathcal C}\cap\{x_1=0\}
        \qquad\text{and}\qquad
        \{1\}\times P_{\mathcal C/1}=P_{\mathcal C}\cap\{x_1=1\}.
    \]
\end{proposition}

Following Bouchet~\cite{Bouchet_1987}, a \emph{delta-matroid} on $[n]$ is a nonempty collection
$\mathcal D\subseteq2^{[n]}$ satisfying symmetric exchange: for $X,Y\in\mathcal D$ and
$x\in X\mathbin\Delta Y$, there is a $y\in X\mathbin\Delta Y$, possibly $y=x$, such that
$X\mathbin\Delta\{x,y\}\in\mathcal D$. Recall also that the type~$B_n$ root directions are
represented by $e_i$ and $e_i\pm e_j$ for $i\ne j$.

\begin{proposition}[\cite{Bouchet_1987}]\label{proposition:char_poly_delta_mat}
    A $0/1$ polytope is the characteristic polytope of a delta-matroid if and only if all of its
    edges are parallel translates of roots of type~$B$, if and only if all of its edges have length
    $1$ or $\sqrt 2$.
\end{proposition}

The \emph{even} delta-matroids, whose feasible sets have the same parity, correspond exactly to
the $0/1$ polytopes with type~$D$ edge directions $e_i\pm e_j$. Thus every result below includes
the type~$D$ case.

\begin{proposition}[{\cite{Bouchet_1987}}]\label{proposition:faces_stable}
    Whenever they are nonempty, contraction and deletion of a delta-matroid are delta-matroids.
    More generally, every nonempty face of a delta-matroid polytope is a delta-matroid polytope.
\end{proposition}

We now characterize type $B$ simplices.
\begin{proposition}[Delta-matroid simplices]\label{proposition:delta_matroid_simplex_volume}
Let $P$ be a delta-matroid polytope which is a simplex. Then
\[
    \nvol(P)\in\{1,2\}.
\]
Moreover, $\nvol(P)=2$ if and only if, up to a permutation of the coordinates, coordinate
reflections $x_i\mapsto 1-x_i$, and deletion of constant coordinates, $P$ is the simplex
\[
    \conv\{\mathbf{0},\,
           e_1+e_2,\,
           e_1+e_3,\,
           e_2+e_3\}.
\]
In particular, every other delta-matroid simplex is unimodular.
\end{proposition}

\begin{proof}
Write $P=\conv\{v_0,v_1,\ldots,v_d\}\subseteq[0,1]^n$.
Choose the vertex $v_0=e_Z$. Reflecting the coordinates indexed by $Z$, namely applying the affine
lattice automorphism
\[
    x_i\longmapsto
    \begin{cases}
        1-x_i,& i\in Z,\\
        x_i,& i\notin Z,
    \end{cases}
\]
sends $v_0$ to $\mathbf{0}$. Its linear part is a diagonal matrix with entries in $\{\pm1\}$, so it
preserves the type~$B$ root system; hence, by
\Cref{proposition:char_poly_delta_mat}, it sends delta-matroid polytopes to delta-matroid polytopes.
It also preserves normalized volume. We may therefore assume that $v_0=\mathbf{0}$.

Write $v_i=e_{X_i}$ for $1\le i\le d$. Since $P$ is a simplex, every pair of its vertices spans an
edge. By \Cref{proposition:char_poly_delta_mat}, every such edge has length $1$ or $\sqrt2$. Thus
\[
    |X_i|\in\{1,2\}
    \qquad\text{and}\qquad
    |X_i\,\Delta\,X_j|\le2
    \quad (i\ne j).
\]
Moreover, the vectors $e_{X_1},\ldots,e_{X_d}$ are linearly independent because
$\mathbf{0},e_{X_1},\ldots,e_{X_d}$ are affinely independent.

We classify the possible family $\{X_1,\ldots,X_d\}$.

\smallskip
\emph{Case 1: at least two of the $X_i$ are singletons.}
Suppose $\{a\}$ and $\{b\}$ occur, with $a\ne b$. If a two-element set $X$ occurs, then
$|X\,\Delta\,\{a\}|\le2$ and $|X\,\Delta\,\{b\}|\le2$ force $a,b\in X$, hence
$X=\{a,b\}$. But then $e_X=e_a+e_b$,
contradicting the linear independence of the vertex vectors. Hence all $X_i$ are singletons.
After deleting unused coordinates and permuting the remaining ones, $P$ is the standard simplex,
so $\nvol(P)=1$.

\smallskip
\emph{Case 2: exactly one of the $X_i$ is a singleton.}
Let it be $\{a\}$. Every two-element set $X_i$ must contain $a$, since otherwise
$|X_i\,\Delta\,\{a\}|=3$. Thus, after relabeling,
\[
    X_1=\{a\},\qquad X_i=\{a,b_i\}\quad(2\le i\le d),
\]
with the $b_i$ distinct. In the coordinate lattice on $\{a,b_2,\ldots,b_d\}$, subtracting the first
column $e_a$ from every other vertex vector transforms
\[
    e_a,\ e_a+e_{b_2},\ldots,e_a+e_{b_d}
\]
into
\[
    e_a,\ e_{b_2},\ldots,e_{b_d}.
\]
This is a unimodular integer column operation, so the vertex vectors form a lattice basis of their
saturated span and $\nvol(P)=1$.

\smallskip
\emph{Case 3: all $X_i$ have cardinality two.}
Regard the $X_i$ as the edges of a simple graph $G$. The condition
$|X_i\,\Delta\,X_j|\le2$ says exactly that every two edges of $G$ intersect. If all edges share a
common vertex $a$, then $G$ is a star, say with edges $\{a,b_1\},\ldots,\{a,b_d\}$. Its vertex
vectors are $e_a+e_{b_1},\ldots,e_a+e_{b_d}$.
Their real span consists, on these coordinates, of the vectors satisfying
\[
    x_a=\sum_{j=1}^d x_{b_j}.
\]
Hence every integral vector in this span is uniquely the integer combination
$\sum_j x_{b_j}(e_a+e_{b_j})$. The generated lattice is therefore saturated, and again
$\nvol(P)=1$.

It remains to consider the case in which the edges of $G$ have no common vertex. Choose two
edges $\{a,b\}$ and $\{a,c\}$ sharing $a$. Since $a$ is not common to all edges, there is an edge
not containing $a$; in order to meet both chosen edges it must be $\{b,c\}$. Any further edge
meeting all three of $\{a,b\}, \{a,c\}$, and $\{b,c\}$
must itself be one of these three. Thus, after deleting unused coordinates, the family is exactly the triangle $\{a,b\}, \{a,c\}, \{b,c\}$.
After relabeling $a,b,c$ as $1,2,3$, the matrix of the three nonzero vertex vectors is
\[
    \begin{pmatrix}
        1&1&0\\
        1&0&1\\
        0&1&1
    \end{pmatrix},
\]
whose determinant has absolute value $2$. Hence $\nvol(P)=2$, and $P$ is precisely the simplex
of \Cref{figure:3-simplex_det_2}.

The converse is immediate from the determinant computation, and coordinate permutations,
coordinate reflections and deletion or insertion of constant coordinates preserve normalized
volume. This proves the classification.
\end{proof}

\begin{remark}
    A simplex in a triangulation of a larger delta-matroid polytope need not itself be a delta-matroid
    polytope, since its edges may be diagonals of the ambient polytope. Thus
    \Cref{proposition:delta_matroid_simplex_volume} does not restrict the simplices in
    \Cref{theorem:main} to volumes $1$ and $2$.
\end{remark}

\subsection{The symmetric resonance arrangement}

A \emph{hyperplane arrangement} in $\RR^n$ is a finite collection of hyperplanes; a \emph{flat} is
an intersection of some of them (the empty intersection being~$\RR^n$).

The \emph{symmetric resonance arrangement} $\mathcal B_n$ consists of the hyperplanes
\[H_{(X^+,X^-)}=\{x\in\RR^n\mid x_{X^+}=x_{X^-}\}\]
indexed by disjoint
$X^+,X^-\subseteq[n]$ with $X^+\cup X^-\ne\emptyset$.

All hyperplanes of $\mathcal B_n$ pass through the origin, so its flats are linear subspaces; taking
$X^-=\emptyset$ recovers the hyperplanes $\{x_X=0\}$ of the classical resonance arrangement, which
plays the same role in type~$A$~\cite{Backman_Liu_2025}.

An affine functional is \emph{$B$-generic} if it is nonconstant on every positive-dimensional flat
of $\mathcal B_n$. Such functionals exist because the nongeneric ones form a finite union of
proper linear subspaces. Moreover, $\ell$ is $B$-generic if and only if
$\ell(c+\,\cdot\,)$ is, since translation does not change its linear part.

We shall only need the following consequence of the inequality description of delta-matroid
polytopes due to Bouchet and Cunningham.

\begin{proposition}[\cite{Bouchet_Cunningham_1995}]\label{proposition:facet_normals}
    Every facet of a delta-matroid polytope $P\subseteq\RR^n$ has an outer normal of the form
    $e_{X^+}-e_{X^-}$ for some disjoint $X^+,X^-\subseteq[n]$. Equivalently, $P$ is cut out by
    inequalities $x_{X^+}-x_{X^-}\le c_{X^+,X^-}$ indexed by disjoint pairs $(X^+,X^-)$.
\end{proposition}

\begin{lemma}\label{lemma:key_1}
    Let $P\subseteq\RR^n$ be a delta-matroid polytope and $F$ a nonempty face of $P$. Then $\dir(F)$
    is a flat of the symmetric resonance arrangement $\mathcal B_n$.
\end{lemma}

\begin{proof}
    Write $P=\{x\mid a_i\cdot x\le b_i,\ i\in I\}$ using the complete ambient
    signed-subset $H$-description of \Cref{proposition:facet_normals}; when $P$ is
    lower-dimensional, equations defining $\aff(P)$ are included as pairs of opposite
    inequalities. Thus each $a_i$ is of the form $e_{X_i^+}-e_{X_i^-}$ with
    $X_i^+,X_i^-$ disjoint. For a nonempty face $F$, setting
    $\Tight(F)=\{i\in I\mid a_i\cdot x=b_i \text{ for all }x\in F\}$, one has
    $\aff(F)=\bigcap_{i\in\Tight(F)}\{x\mid a_i\cdot x=b_i\}$; see~\cite[Section~2]{Ziegler_1995}.
    Translating by any point of $F$ turns each of these hyperplanes into
    $\{x_{X_i^+}-x_{X_i^-}=0\}=H_{(X_i^+,X_i^-)}\in\mathcal B_n$, so $\dir(F)$ is an intersection of
    hyperplanes of $\mathcal B_n$, i.e.\ a flat.
\end{proof}

\subsection{Regular subdivisions}

We use the standard terminology of~\cite[Chapter~2]{DeLoera_Rambau_Santos_2010}. A
\emph{subdivision} $\mathcal S$ of $P$ is a face-to-face polyhedral complex with union $P$; it is a
\emph{triangulation} if its cells are simplices. For a finite
$A\subseteq P\cap\ZZ^n$ with $\conv(A)=P$, a subdivision with vertices in $A$ is \emph{regular} if
some weight $f\colon A\to\RR$ makes its cells the projections of the lower faces of
$\conv\{(v,f(v))\mid v\in A\}$. Equivalently, $S$ is a cell when an affine function $h$ satisfies
$h=f$ on $A\cap S$ and $h<f$ on $A\setminus S$; we then say that $f$ \emph{induces} the
subdivision.

We use freely the following two standard facts, for which we refer
to~\cite[Sections~2.3 and~2.5]{DeLoera_Rambau_Santos_2010}: the restriction of a regular subdivision of
$P$ to a face of $P$ is the regular subdivision of that face induced by the restricted function;
and if $f$ induces $\mathcal S$ and $h\colon A\to\RR$ is arbitrary, then for all sufficiently small
$\varepsilon>0$ the function $f+\varepsilon h$ induces a subdivision refining $\mathcal S$ whose
restriction to each cell $S\in\mathcal S$ is the subdivision induced by $h|_{S\cap A}$. In
particular, the resulting subdivision does not depend on $\varepsilon$ once $\varepsilon$ is small
enough, as one sees from the existence of the secondary fan.

We shall use the following uniform version when the refining function itself depends on the
perturbation parameter.

\begin{lemma}[Uniform small perturbations]\label{lemma:uniform_perturbation}
Let $A\subseteq\RR^n$ be a finite point configuration in convex position (every point of $A$ is a
vertex of $\conv(A)$), let
$f\colon A\to\RR$ induce a regular subdivision $\mathcal S$, and let
$\mathcal H\subseteq\RR^A$ be bounded. There exists
$\delta_0>0$ such that, for every $h\in\mathcal H$ and every $0<\delta<\delta_0$, the function
$f+\delta h$ induces a refinement of $\mathcal S$. On each cell $S\in\mathcal S$, the restriction
of this refinement is the regular subdivision induced by $h|_{A\cap S}$.
\end{lemma}

\begin{proof}
Work in the space of weights on $A$ modulo restrictions of affine functions. The weights inducing
$\mathcal S$ form the relative interior of a cone $\sigma$ of the secondary fan, and refinements of
$\mathcal S$ correspond to cones containing $\sigma$ as a face. Since the secondary fan is finite
and the class of $f$ lies in $\relint(\sigma)$, a neighborhood of the class of $f$
is contained in the union of those cones. The image of the bounded set $\mathcal H$ is bounded, so
$f+\delta h$ lies in this neighborhood for all $h\in\mathcal H$ once $\delta>0$ is sufficiently
small. Hence the induced subdivision refines $\mathcal S$.

For a cell $S\in\mathcal S$, the restriction $f|_{A\cap S}$ agrees with an affine function.
Consequently $(f+\delta h)|_{A\cap S}$ differs from $\delta h|_{A\cap S}$ by an affine function,
and the two weights induce the same subdivision of~$S$.
\end{proof}

\section{Dyadic systems, simplices and triangulations}\label{section:dyadic}

\subsection{Dyadic numbers and totally dyadic systems}

The ring $\DD$ of \emph{dyadic numbers} is the subring $\ZZ[1/2]$ of $\QQ$ of elements $p/2^k$ with
$p\in\ZZ$ and $k\in\NN$.

\begin{definition}
    Let $V\subseteq\RR^n$ be a rational subspace. A set $D$ of integer vectors of $V$ is a
    \emph{dyadic generating set} of $\Lambda_V=V\cap\ZZ^n$ if every $v\in\Lambda_V$ is a
    \emph{dyadic combination} of $D$, that is, $v=\sum_{d\in D}\alpha_d d$ with all
    $\alpha_d\in\DD$. A dyadic generating set which is moreover a basis of $V$ is a \emph{dyadic
    basis} of $\Lambda_V$.
\end{definition}

\begin{definition}[Totally dyadic system]
    A set $D$ of integer vectors of $\RR^n$ is a \emph{totally dyadic system} if every subset
    $E\subseteq D$ is a dyadic generating set of $\lin(E)\cap\ZZ^n$.
\end{definition}

\begin{lemma}\label{lemma:det_power_2_TDS}
    A set $D$ of integer vectors is a totally dyadic system if and only if, for every linearly
    independent subset $E\subseteq D$, the greatest common divisor of the maximal (i.e.\
    $|E|\times|E|$) minors of the matrix with columns $E$ is a power of two.
\end{lemma}

\begin{proof}
    Fix a linearly independent $E=\{d_1,\ldots,d_r\}\subseteq D$ and let $M$ be the matrix with
    columns $d_1,\ldots,d_r$. The saturation $\lin(E)\cap\ZZ^n$ contains the column lattice
    $\ZZ\langle E\rangle$ with finite index, and this index is the order of the torsion subgroup of
    $\coker(M)$. By Smith normal form it is the greatest common divisor of the $r\times r$ minors of
    $M$. Hence $E$ dyadically generates $\lin(E)\cap\ZZ^n$ if and only if this index is a power of
    two.

    If $D$ is a totally dyadic system, this applies to every independent subset. Conversely, assume
    the stated gcd condition for every linearly independent subset of $D$, and let $E\subseteq D$
    be arbitrary. Choose a maximal linearly independent subset $E_0\subseteq E$. Then
    $\lin(E_0)=\lin(E)$, and the hypothesis shows that $E_0$ dyadically generates
    $\lin(E)\cap\ZZ^n$. Since $E_0\subseteq E$, the set $E$ does as well. Thus $D$ is totally dyadic.
\end{proof}

We record the totally dyadic system relevant to type~$B$. Let
\[
    B^+\;\coloneqq\;\{e_i\}_{i\in[n]}\;\cup\;\{e_i-e_j\}_{1\le i<j\le n}\;\cup\;\{e_i+e_j\}_{1\le
    i<j\le n}
\]
be a set of representatives of the edge directions of delta-matroid polytopes, one per line spanned
by a type~$B$ root.

\begin{lemma}\label{lemma:type_B_TDS}
    The set $B^+$ is a totally dyadic system.
\end{lemma}

\begin{proof}
    It is enough to show that every square minor of the matrix with columns $B^+$ is either zero or
    a signed power of two. Consider an $r\times r$ submatrix $M$. Every column of $M$ has at most two
    nonzero entries, each equal to $\pm1$. If a row or column has at most one nonzero entry, expansion
    along it reduces the determinant, up to sign, to a smaller such minor; induction applies.

    Otherwise every column has exactly two nonzero entries and every row has at least two. Since
    the square matrix has exactly $2r$ nonzero entries, every row has exactly two as well.
    Interpreting the columns as signed edges on the selected rows, the corresponding $2$-regular
    multigraph is a disjoint union of cycles.
    After permuting rows and columns, $M$ is block diagonal with one block per cycle. For a cycle
    block, expansion gives determinant either $0$ or $\pm2$. Thus $\det(M)$ is $0$ or $\pm2^k$ for
    some $k$. The conclusion follows from \Cref{lemma:det_power_2_TDS}.
\end{proof}

\begin{remark}[Gcd condition versus individual minors]\label{remark:equimodular}
    The gcd criterion of \Cref{lemma:det_power_2_TDS} is strictly weaker than requiring every minor
    to be $0$ or $\pm$ a power of two: it constrains only the gcd of the maximal minors of each
    independent subset. The proof of \Cref{lemma:type_B_TDS} shows that $B^+$ does satisfy the
    stronger property. We only use the gcd criterion below.
\end{remark}

\subsection{Dyadic simplices}

\begin{definition}[Dyadic simplices and triangulations]
    A lattice simplex is \emph{dyadic} if its normalized volume is a power of two. A lattice
    triangulation $\mathcal T$ is \emph{dyadic} if every maximal simplex is dyadic; we denote its
    set of maximal simplices by $\mathcal T_{\max}$. By \Cref{lemma:dyadic_faces}, equivalently
    every simplex of $\mathcal T$ is dyadic.
\end{definition}

The next lemma is the bridge between the volume definition and the lattice statements that drive
the induction; it is the type~$B$ counterpart of the characterization of unimodular simplices used
in~\cite[Definition~2.10]{Backman_Liu_2025}.

\begin{lemma}\label{lemma:dyadic_bridge}
    Let $T=\conv(v_0,\ldots,v_d)$ be a lattice simplex and $\Lambda=\dir(T)\cap\ZZ^n$ the saturated
    lattice of its direction space. The following are equivalent.
    \begin{compactenum}
        \item $T$ is dyadic.
        \item The edge vectors $\{v_i-v_0\mid 1\le i\le d\}$ form a dyadic basis of $\Lambda$.
        \item For some, equivalently every, vertex $v_k$, the edge vectors $\{v_i-v_k\mid i\ne k\}$
              form a dyadic basis of $\Lambda$.
        \item For some, equivalently every, spanning tree $\tau$ of the complete graph on
              $\{0,\ldots,d\}$, the $d$ edge vectors $\{v_i-v_j\mid ij\in\tau\}$ form a dyadic basis
              of $\Lambda$.
    \end{compactenum}
\end{lemma}

\begin{proof}
    The normalized volume of $T$ relative to $\Lambda$ equals the index
    $[\Lambda:\ZZ\langle v_i-v_0\rangle_{i\ge1}]$, that is, the order of the finite abelian group
    $Q\coloneqq\Lambda/\ZZ\langle v_i-v_0\rangle_{i\ge1}$. The edge vectors dyadically generate
    $\Lambda$ if and only if every element of $Q$ is annihilated by a power of two, i.e.\ if and
    only if $Q$ is a $2$-group, i.e.\ if and only if $|Q|$ is a power of two. This proves
    $(1)\Leftrightarrow(2)$; and since the normalized volume is intrinsic to $(T,\Lambda)$ and does
    not depend on the chosen apex, $(2)\Leftrightarrow(3)$.

    For $(3)\Rightarrow(4)$, let $\tau$ be a spanning tree of the complete graph on
    $\{0,\ldots,d\}$. Each $\tau$-edge vector $v_i-v_j$ equals $(v_i-v_0)-(v_j-v_0)$, so the
    $\tau$-edge vectors are obtained from the $v_0$-star by an integer linear change of variables;
    conversely, since $\tau$ is a spanning tree, each $v_i-v_0$ is the signed sum of the
    $\tau$-edges along the unique $v_i$--$v_0$ path in $\tau$, so this change of variables lies in
    $\GL_d(\ZZ)$. It therefore does not change the generated lattice, hence not the quotient
    $Q$ either. Condition (3) is the special case of (4) where $\tau$ is a star, giving
    $(4)\Rightarrow(3)$.
\end{proof}

\begin{lemma}[Faces of dyadic simplices]\label{lemma:dyadic_faces}
    Every face of a dyadic lattice simplex is dyadic.
\end{lemma}

\begin{proof}
    Let $T$ be dyadic and let $F$ be a nonempty face. Choose a vertex $v_0$ of $F$ and write
    \[
        \Gamma_T=\ZZ\langle v-v_0\mid v\in\vertt(T)\setminus\{v_0\}\rangle,
        \qquad \Lambda=\dir(T)\cap\ZZ^n.
    \]
    By \Cref{lemma:dyadic_bridge}, $\Lambda/\Gamma_T$ is a finite $2$-group. Set
    $V_F=\dir(F)$, $\Lambda_F=V_F\cap\ZZ^n$, and
    $\Gamma_F=\ZZ\langle v-v_0\mid v\in\vertt(F)\setminus\{v_0\}\rangle$. Since the full edge star at
    $v_0$ is a real basis of $\dir(T)$, one has $\Gamma_T\cap V_F=\Gamma_F$. Therefore the inclusion
    $\Lambda_F\hookrightarrow\Lambda$ induces an injection
    $\Lambda_F/\Gamma_F\hookrightarrow\Lambda/\Gamma_T$. Hence $\Lambda_F/\Gamma_F$ is a finite $2$-group, so
    $\nvol(F)$ is a power of two by \Cref{lemma:dyadic_bridge}.
\end{proof}

Being dyadic is a genuine restriction, already among $0/1$ simplices.

\begin{example}\label{example:non_dyadic}
    In $\RR^4$, the $0/1$ simplex
    $T=\conv\{\mathbf{0},\,e_{\{1,2,3\}},\,e_{\{1,2,4\}},\,e_{\{1,3,4\}},\,e_{\{2,3,4\}}\}$ has
    normalized volume $|\det(J-R)|=3$, where $J$ is the all-ones $4\times 4$ matrix and $R$ the
    antidiagonal permutation matrix; so $T$ is not dyadic. Consistently with
    \Cref{theorem:main}, $T$ is not a delta-matroid polytope: its edge vector
    $e_{\{1,2,3\}}-\mathbf{0}$ is not a root of type~$B$, and indeed the symmetric exchange axiom
    fails for $X=\emptyset$, $Y=\{1,2,3\}$ and $x=1$.
\end{example}

\subsection{Dyadic complementarity}

Recall that two linear subspaces $V,W\subseteq\RR^n$ are \emph{independent} if $V\cap W=\{0\}$.

\begin{definition}
    Rational linear subspaces $V,W\subseteq\RR^n$ are \emph{dyadic-complementary} if there exist
    $E_V\subseteq V\cap\ZZ^n$ and $E_W\subseteq W\cap\ZZ^n$ which together dyadically generate
    $(V+W)\cap\ZZ^n$.
\end{definition}

The following reformulation is what makes the whole section work: for independent subspaces, dyadic
complementarity is a property of the pair $(V,W)$ alone, expressed by a single index, and it is
then automatically witnessed by \emph{every} choice of dyadic generating sets.

\begin{proposition}\label{proposition:dyadic_compl_index}
    Let $V,W\subseteq\RR^n$ be independent rational subspaces and write $\Lambda_V=V\cap\ZZ^n$,
    $\Lambda_W=W\cap\ZZ^n$ and $\Lambda=(V+W)\cap\ZZ^n$. The following are equivalent.
    \begin{compactenum}
        \item $V$ and $W$ are dyadic-complementary.
        \item The index $[\Lambda:\Lambda_V\oplus\Lambda_W]$ is a power of two.
        \item For \emph{all} dyadic generating sets $E_V$ of $\Lambda_V$ and $E_W$ of
              $\Lambda_W$, the union $E_V\cup E_W$ dyadically generates~$\Lambda$.
    \end{compactenum}
\end{proposition}

\begin{proof}
    Since $V$ and $W$ are independent, $\Lambda_V\oplus\Lambda_W$ is a sublattice of $\Lambda$ of
    full rank $\dim V+\dim W$, so the index in (2) is finite.

    $(1)\Rightarrow(2)$. Let $E_V^0\subseteq\Lambda_V$ and $E_W^0\subseteq\Lambda_W$ witness dyadic
    complementarity. Then every element of
    $\Lambda/(\ZZ\langle E_V^0\rangle+\ZZ\langle E_W^0\rangle)$
    has $2$-power order; being finitely generated, this group is a finite $2$-group, say of exponent
    $2^k$. Since $\ZZ\langle E_V^0\rangle+\ZZ\langle E_W^0\rangle$ is contained in
    $\Lambda_V\oplus\Lambda_W$, we get $2^k\Lambda\subseteq\Lambda_V\oplus\Lambda_W$, so
    $\Lambda/(\Lambda_V\oplus\Lambda_W)$ has $2$-power order.

    $(2)\Rightarrow(3)$. Let $2^k$ be the exponent of $\Lambda/(\Lambda_V\oplus\Lambda_W)$, let
    $E_V,E_W$ be dyadic generating sets of $\Lambda_V,\Lambda_W$, and let $x\in\Lambda$. Then
    $2^kx\in\Lambda_V\oplus\Lambda_W$; write $2^kx=x_V+x_W$ with $x_V\in\Lambda_V$,
    $x_W\in\Lambda_W$. Then $x_V$ is a $\DD$-combination of $E_V$ and $x_W$ one of $E_W$; dividing
    by $2^k$ exhibits $x$ as a $\DD$-combination of $E_V\cup E_W$.

    $(3)\Rightarrow(1)$. A $\ZZ$-basis of $\Lambda_V$ is in particular a dyadic generating set, and
    likewise for $\Lambda_W$; apply (3) to these.
\end{proof}

\begin{lemma}\label{lemma:dyadic_compl_subspace}
    Let $V,W\subseteq\RR^n$ be independent dyadic-complementary rational subspaces. If
    $V'\subseteq V$ and $W'\subseteq W$ are rational subspaces, then $V'$ and $W'$ are independent and
    dyadic-complementary.
\end{lemma}

\begin{proof}
    Independence is clear. Let $2^k$ be the exponent of $\Lambda/(\Lambda_V\oplus\Lambda_W)$, which
    is a $2$-group by \Cref{proposition:dyadic_compl_index}, so that
    $2^k\Lambda\subseteq\Lambda_V\oplus\Lambda_W$. Let
    $z\in(V'+W')\cap\ZZ^n\subseteq\Lambda$ and write $z=u+w$ with $u\in V'$ and $w\in W'$. Then
    $2^kz=2^ku+2^kw\in\Lambda_V\oplus\Lambda_W$, and uniqueness of the decomposition along
    $V\oplus W$ gives $2^ku\in V'\cap\ZZ^n$ and $2^kw\in W'\cap\ZZ^n$. Hence
    $2^k((V'+W')\cap\ZZ^n)\subseteq(V'\cap\ZZ^n)\oplus(W'\cap\ZZ^n)$. The corresponding quotient
    is therefore a finite $2$-group; conclude with \Cref{proposition:dyadic_compl_index}.
\end{proof}

\begin{lemma}\label{lemma:dyadic_compl_poly}
    Let $P$ and $Q$ be delta-matroid polytopes in $\RR^n$. Then $\dir(P)$ and $\dir(Q)$ are
    dyadic-complementary.
\end{lemma}

\begin{proof}
    Let $D_P$ and $D_Q$ be the sets of edge directions of $P$ and $Q$. By
    \Cref{proposition:char_poly_delta_mat} they may be chosen inside $B^+$, and they span $\dir(P)$
    and $\dir(Q)$ respectively. By \Cref{lemma:type_B_TDS} the set $B^+$ is a totally dyadic system,
    so $D_P\cup D_Q\subseteq B^+$ dyadically generates
    $\lin(D_P\cup D_Q)\cap\ZZ^n=(\dir(P)+\dir(Q))\cap\ZZ^n$, with
    $D_P\subseteq\dir(P)\cap\ZZ^n$ and $D_Q\subseteq\dir(Q)\cap\ZZ^n$. This is exactly dyadic
    complementarity.
\end{proof}

\begin{remark}\label{remark:embedding}
    All the notions above are compatible with the standard embedding
    $\RR^{n-1}\hookrightarrow\RR^n$, $y\mapsto(0,y)$: for a rational subspace $V\subseteq\RR^{n-1}$
    one has $(\{0\}\times V)\cap\ZZ^n=\{0\}\times(V\cap\ZZ^{n-1})$, so $V,W$ are
    dyadic-complementary in $\RR^{n-1}$ if and only if $\{0\}\times V$ and $\{0\}\times W$ are
    dyadic-complementary in $\RR^n$; and a lattice simplex of $\RR^{n-1}$ is dyadic if and only if its image is.
    We use this silently below.
\end{remark}

\subsection{The join lemma}

The following lemma isolates the geometric heart of the induction: dyadicness is preserved by
joining two dyadic simplices placed in consecutive integer slices, provided their direction spaces
are independent and dyadic-complementary.

\begin{lemma}[Join lemma]\label{lemma:dyadic_join}
    Let $T_0\subseteq\{x\in\RR^n\mid x_1=0\}$ and $T_1\subseteq\{x\in\RR^n\mid x_1=1\}$ be dyadic
    lattice simplices whose direction spaces $\dir(T_0)$ and $\dir(T_1)$ are independent and
    dyadic-complementary. Then $T\coloneqq\conv(T_0\cup T_1)$ is a dyadic lattice simplex with
    $\vertt(T)=\vertt(T_0)\sqcup\vertt(T_1)$.
\end{lemma}

\begin{proof}
    Fix base vertices $w_0\in\vertt(T_0)$ and $w_1\in\vertt(T_1)$, and set $v\coloneqq w_1-w_0$, an
    integer vector with first coordinate $1$. Since $\aff(T)$ is the affine hull of
    $\aff(T_0)\cup\aff(T_1)$, we have $\dir(T)=\dir(T_0)+\dir(T_1)+\RR v$. This sum is direct:
    if $u_0+u_1+\alpha v=0$ with $u_i\in\dir(T_i)$, then comparing first coordinates gives
    $\alpha=0$, since $\dir(T_0),\dir(T_1)\subseteq\{x_1=0\}$; then $u_0=u_1=0$ by independence.
    Hence
    \begin{equation}\label{equation:join_decomposition}
        \dir(T)=\dir(T_0)\oplus\dir(T_1)\oplus\RR v,
    \end{equation}
    so $\dim T=\dim T_0+\dim T_1+1$, whereas
    $|\vertt(T_0)|+|\vertt(T_1)|=\dim T_0+\dim T_1+2=\dim T+1$. Therefore
    $\vertt(T_0)\cup\vertt(T_1)$ is affinely independent and $T$ is a simplex with that vertex set.

    \emph{$T$ is dyadic.} Let $\beta_0=\{u-w_0\mid u\in\vertt(T_0)\setminus\{w_0\}\}$ and
    $\beta_1=\{u-w_1\mid u\in\vertt(T_1)\setminus\{w_1\}\}$ be the two edge stars. By
    \Cref{lemma:dyadic_bridge} applied to $T_0$ and $T_1$, the set $\beta_i$ is a dyadic basis of
    $\dir(T_i)\cap\ZZ^n$; since $\dir(T_0)$ and $\dir(T_1)$ are dyadic-complementary,
    \Cref{proposition:dyadic_compl_index}(3) shows that $\beta_0\cup\beta_1$ dyadically generates
    $(\dir(T_0)+\dir(T_1))\cap\ZZ^n$. Now let $x\in\dir(T)\cap\ZZ^n$ and decompose it according
    to~\eqref{equation:join_decomposition} as $x=u+\alpha v$ with
    $u\in\dir(T_0)\oplus\dir(T_1)\subseteq\{x_1=0\}$. Comparing first coordinates gives
    $\alpha=x_1\in\ZZ$, so $u=x-x_1v$ is an integer vector, hence a dyadic combination of
    $\beta_0\cup\beta_1$, and $x$ is a dyadic combination of $\beta\coloneqq\beta_0\cup\beta_1\cup\{v\}$.
    Finally $\beta$ consists of the star of $T_0$ at $w_0$, the star of $T_1$ at $w_1$, and the
    bridge $v=w_1-w_0$; these are the edges of a spanning tree of the complete graph on
    $\vertt(T)$, and $|\beta|=\dim T$. By \Cref{lemma:dyadic_bridge}(4), $T$ is dyadic.

\end{proof}

\section{Regular dyadic triangulations}\label{section:main}

\subsection{The main theorem}

Throughout this section, for each $j\in[n-1]$ fix a $B$-generic affine functional
$\ell_j\colon\RR^{\{j+1,\ldots,n\}}\to\RR$; the preceding genericity observation guarantees that
these exist. For $\varepsilon>0$, define
$\varphi^{(n)}\coloneqq 0$ and, downwards for $j=n-1,\ldots,1$,
\begin{equation}\label{equation:the_function}
    \varphi^{(j)}(x_j,\ldots,x_n)\;\coloneqq\;x_j\,\ell_j(x_{j+1},\ldots,x_n)
    \;+\;\varepsilon\,\varphi^{(j+1)}(x_{j+1},\ldots,x_n),
\end{equation}
and we set $\varphi\coloneqq\varphi^{(1)}=\sum_{j=1}^{n-1}\varepsilon^{\,j-1}
x_j\,\ell_j(x_{j+1},\ldots,x_n)$. Note that $\varphi$ is defined on all of $\RR^n$, not merely on
the vertex set of a given polytope; this uniformity is what will let us pass from delta-matroid
polytopes to arbitrary integral type~$B$ generalized permutohedra in
\Cref{corollary:typeB_GP}.

\begin{theorem}\label{theorem:main}
    Every delta-matroid polytope admits a regular dyadic triangulation. More precisely, after fixing
    the $B$-generic affine functionals $\ell_j$, there exists $\varepsilon_n>0$ such that for every
    $0<\varepsilon<\varepsilon_n$ the function $\varphi$ of~\eqref{equation:the_function} induces a
    dyadic triangulation of every delta-matroid polytope in $\RR^n$.
\end{theorem}

\begin{proof}
    We argue by induction on $n$. If $n=1$, a delta-matroid polytope is either a point or the unit
    segment $[0,1]$, and in both cases the trivial triangulation is regular and dyadic; take
    $\varepsilon_1=1$.

    Let $n\ge 2$. By the induction hypothesis, let $\varepsilon_{n-1}>0$ be a common threshold in
    dimension $n-1$, and for now assume $0<\varepsilon<\min\{1,\varepsilon_{n-1}\}$. Let
    $\mathcal D$ be a delta-matroid on $[n]$ and $P\coloneqq P_{\mathcal D}\subseteq\RR^n$ its
    polytope; recall $\vertt(P)=P\cap\ZZ^n\subseteq\{0,1\}^n$. Define
    $P_0,P_1\subseteq\RR^{n-1}$ by
    \[
        \{0\}\times P_0=P\cap\{x_1=0\},\qquad \{1\}\times P_1=P\cap\{x_1=1\}.
    \]
    By \Cref{proposition:polytope_contr_dele}, $P_0$ and $P_1$ are, when nonempty, the polytopes of
    the delta-matroids $\mathcal D\setminus 1$ and $\mathcal D/1$; either may be empty, which happens
    exactly when $1$ is a coloop or a loop.

    Write $y=(x_2,\ldots,x_n)$, $\ell\coloneqq\ell_1$ and $\varphi'\coloneqq\varphi^{(2)}$, so that
    $\varphi(x)=x_1\ell(y)+\varepsilon\varphi'(y)$. By the induction hypothesis, $\varphi'$ induces
    a dyadic triangulation of every delta-matroid polytope of $\RR^{n-1}$, in particular each
    nonempty one among $P_0$ and $P_1$.

    If $P_1=\emptyset$ then $\varphi$ restricted to $\vertt(P)$ equals $\varepsilon\varphi'$ and
    $P=\{0\}\times P_0$, so we are done; symmetrically if $P_0=\emptyset$, since
    then $\varphi=\ell(y)+\varepsilon\varphi'(y)$ on $\vertt(P)$ differs from
    $\varepsilon\varphi'$ by the affine function $\ell$, which does not change the induced
    subdivision. Assume from now on that $P_0,P_1\ne\emptyset$.

    \smallskip
    \emph{Step 1: the coarse subdivision.}
    Let $g(x)\coloneqq x_1\ell(y)$ and let $\mathcal S$ be the subdivision of $P$ induced by
    $g|_{\vertt(P)}$. Since $\vertt(P)\subseteq\{x_1=0\}\cup\{x_1=1\}$, the lifted polytope is
    \[
        \widehat P=\conv\Big(\big\{(0,y,0)\mid y\in\vertt(P_0)\big\}\cup
                             \big\{(1,y,\ell(y))\mid y\in\vertt(P_1)\big\}\Big).
    \]
    A lower face of $\widehat P$ is exposed by an affine functional, whose restriction to
    $\{x_1=0\}$ (resp.\ $\{x_1=1\}$) exposes a face of $P_0$ (resp.\ of $P_1$). Hence every cell of
    $\mathcal S$ is of the form
    \[
        S=\conv\big(\{0\}\times F_0\ \cup\ \{1\}\times F_1\big)
    \]
    with $F_0$ a face of $P_0$ and $F_1$ a face of $P_1$, possibly empty. By
    \Cref{proposition:faces_stable}, each nonempty $F_i$ is a delta-matroid polytope.

    \smallskip
    \emph{Step 2: independence, via genericity.}
    Fix a cell $S=\conv(\{0\}\times F_0\cup\{1\}\times F_1)$ of $\mathcal S$ with $F_0,F_1\ne\emptyset$,
    and set $L\coloneqq\dir(F_0)\cap\dir(F_1)$. By \Cref{lemma:key_1}, $\dir(F_0)$ and $\dir(F_1)$
    are flats of $\mathcal B_{n-1}$, hence so is $L$.

    Since $S$ is a cell of $\mathcal S$, there is an affine function $\widetilde g\colon\RR^n\to\RR$
    with $\widetilde g=g$ on $\vertt(P)\cap S$ and $\widetilde g<g$ on $\vertt(P)\setminus S$; write
    $\widetilde g(x_1,y)=bx_1+a\cdot y+c$. On $\vertt(F_0)$ we have $g=0$, so $a\cdot y+c=0$ there,
    hence on all of $\aff(F_0)$; in particular $a\cdot u=0$ for every $u\in\dir(F_0)$. On
    $\vertt(F_1)$ we have $g=\ell$, so $\ell(y)=a\cdot y+b+c$ there, hence on all of $\aff(F_1)$;
    in particular the linear part of $\ell$ agrees with $u\mapsto a\cdot u$ on $\dir(F_1)$.
    Therefore the linear part of $\ell$ vanishes on $L$, i.e.\ $\ell$ is constant on the flat $L$.
    As $\ell$ is $B$-generic, $L$ is not positive-dimensional, so $L=\{0\}$: the subspaces
    $\dir(F_0)$ and $\dir(F_1)$ are independent. By \Cref{lemma:dyadic_compl_poly} they are also
    dyadic-complementary, and by \Cref{remark:embedding} the same holds for
    $\dir(\{0\}\times F_0)$ and $\dir(\{1\}\times F_1)$ inside $\RR^n$.

    \smallskip
    \emph{Step 3: the refinement.}
    It remains to make the additional smallness condition on $\varepsilon$ uniform in $P$. For fixed
    $n$ there are only finitely many delta-matroids on $[n]$, hence finitely many vertex
    configurations and finitely many coarse subdivisions $\mathcal S$ occurring above. On these
    finite vertex sets the family of weights
    $\{\varphi'(\varepsilon):0<\varepsilon\le1\}$ is bounded. Applying
    \Cref{lemma:uniform_perturbation} to each of the finitely many configurations and taking the
    minimum of the resulting thresholds gives a common $\eta_n>0$ such that, for every such $P$ and
    every $0<\varepsilon<\eta_n$, the function
    $g+\varepsilon\varphi'(\varepsilon)$ refines the coarse subdivision and restricts on each coarse
    cell to the subdivision induced by $\varphi'(\varepsilon)$. Taking
    $\varepsilon_n=\min\{1,\varepsilon_{n-1},\eta_n\}$ makes both the induction hypothesis and the
    refinement statement valid simultaneously for every delta-matroid polytope in $\RR^n$.

    Let $0<\varepsilon<\varepsilon_n$ and let $\mathcal T$ be the subdivision induced by
    $\varphi=g+\varepsilon\varphi'$. Its traces on $\{0\}\times P_0$ and $\{1\}\times P_1$ are the
    dyadic triangulations supplied by the induction hypothesis.

    Let $T$ be a cell of $\mathcal T$, contained in a cell
    $S=\conv(\{0\}\times F_0\cup\{1\}\times F_1)$ of $\mathcal S$. Split the vertices of $T$ by
    their first coordinate and write
    \[
        T=\conv(\{0\}\times T_0\cup\{1\}\times T_1),
    \]
    where either $T_i$ may be empty. If $T_i$ is nonempty, it is a cell of the trace of
    $\mathcal T$ on $F_i$, hence a dyadic simplex. If one of $T_0,T_1$ is empty, then $T$ is simply
    a dyadic simplex in the other slice. If both are nonempty, then necessarily $F_0,F_1$ are
    nonempty; by Step~2 their direction spaces are independent and dyadic-complementary, and
    \Cref{lemma:dyadic_compl_subspace} gives the same properties for $\dir(T_0)$ and $\dir(T_1)$.
    The join lemma, \Cref{lemma:dyadic_join}, therefore shows that $T$ is dyadic.

    Thus every maximal cell of $\mathcal T$ is a dyadic simplex, and $\mathcal T$ is a regular dyadic triangulation of~$P$.
\end{proof}

\subsection{Integral type \texorpdfstring{$B$}{B} generalized permutohedra}

An \emph{integral type~$B$ generalized permutohedron} is a lattice polytope whose edge directions
are type~$B$ roots, equivalently whose normal fan coarsens the type~$B$ Coxeter fan
~\cite{Ardila_Castillo_Eur_Postnikov_2020}. Its $0/1$ members are exactly the delta-matroid
polytopes. Because $B_n$ and $C_n$ have the same root directions and Coxeter fan, this is also the
type~$C$ class when the coordinate lattice remains $\ZZ^n$.

We dice by the short-root hyperplanes $x_i=k$ for $i\in[n]$ and $k\in\ZZ$. This is the standard
lattice dicing of~\cite[Section~2.1.3]{Haase_Paffenholz_Piechnik_Santos_2021}: once its cells are
lattice polytopes, the quadratic weight $q(x)=\sum_i x_i^2$ induces the subdivision.

Haase--Paffenholz--Piechnik--Santos also use the same short-root dicing for polytopes whose
\emph{facet normals} lie in $B_n$~\cite[Section~3.1.2]{Haase_Paffenholz_Piechnik_Santos_2021}.
Our hypothesis is different: the \emph{edge directions} of $P$ lie in $B_n$. The precise
preservation result needed here is due to Eur, Fink, Larson and
Spink~\cite[Proposition~2.12]{Eur_Fink_Larson_Spink_2024}; they also point to earlier proofs
using bisubmodular functions.

\begin{lemma}[Coordinate dicing]\label{lemma:dicing}
    Let $P\subseteq\RR^n$ be an integral type~$B$ generalized permutohedron and
    $C_c\coloneqq c+[0,1]^n$ with $c\in\ZZ^n$. If $P\cap C_c\ne\emptyset$, then
    $P\cap C_c=c+P_{\mathcal D}$ for some delta-matroid $\mathcal D$ on $[n]$.
\end{lemma}

\begin{proof}
    Apply \cite[Proposition~2.12]{Eur_Fink_Larson_Spink_2024} to $P-c$, and translate back by $c$.
\end{proof}

\begin{corollary}\label{corollary:typeB_GP}
    Every integral type~$B$ generalized permutohedron, with integrality taken in $\ZZ^n$,
    admits a regular dyadic triangulation. At the level of root directions and Coxeter fans, this
    is also a type~$C$ statement with the same lattice convention.
    In particular, so does every integral type~$D$ generalized permutohedron.
\end{corollary}

\begin{proof}
    Let $P\subseteq\RR^n$ be such a polytope. After an integer translation we may assume
    $P\subseteq[0,R]^n$ for some $R\in\NN$. Dice $P$ by all integer coordinate hyperplanes
    $x_i=k$. By \Cref{lemma:dicing}, every nonempty cell
    $S_c\coloneqq P\cap(c+[0,1]^n)$, with $c\in\{0,\ldots,R-1\}^n$, is a translate of a delta-matroid
    polytope. Hence~\cite[Section~2.1.3]{Haase_Paffenholz_Piechnik_Santos_2021} applies: the
    restriction of $q(x)=\sum_i x_i^2$ to $P\cap\ZZ^n$ induces the regular subdivision
    $\mathcal S_{\mathrm{dice}}$ into these intersections and their faces.

    Choose $0<\varepsilon<\varepsilon_n$ as in \Cref{theorem:main}, let $\varphi$ be the
    corresponding function~\eqref{equation:the_function}, and let $\mathcal T$ be the subdivision
    induced by $q+\delta\varphi$ for $\delta>0$ small enough, so that $\mathcal T$
    refines $\mathcal S_{\mathrm{dice}}$ and restricts on each cell $S_c$ to the subdivision induced
    by $\varphi|_{S_c}$. Writing $x=c+z$ with $z\in[0,1]^n$, we have for each $j$
    \[
        x_j\,\ell_j(x_{j+1},\ldots,x_n)
        \;=\;\underbrace{c_j\,\ell_j(c_{j+1}+z_{j+1},\ldots,c_n+z_n)}_{\text{affine in }z}
        \;+\;z_j\,\widetilde\ell_j(z_{j+1},\ldots,z_n),
    \]
    where $\widetilde\ell_j\coloneqq\ell_j(c_{j+1}+\,\cdot\,,\ldots,c_n+\,\cdot\,)$ is again
    $B$-generic because translation leaves the linear part unchanged. Moreover
    $\widetilde\ell_j-\ell_j$ is constant, so replacing every $\ell_j$ by
    $\widetilde\ell_j$ changes the function~\eqref{equation:the_function} only by an affine
    function of $z$. Hence the same value of $\varepsilon$ induces on the delta-matroid polytope
    $S_c-c$ exactly the triangulation covered by \Cref{theorem:main}. Therefore
    $\varphi|_{S_c}$ induces a dyadic triangulation of $S_c$.

    Since all these cell triangulations arise from the single global function $q+\delta\varphi$,
    they automatically agree on shared faces, and $\mathcal T$ is a regular dyadic
    triangulation of $P$. The type~$D$ statement follows because every type~$D$ root is a type~$B$
    root.
\end{proof}

\subsection{The dyadic decomposition property}

Recall that a lattice polytope $P$ has the \emph{integer decomposition property} (IDP) if every
$x\in kP\cap\ZZ^n$ is the sum of $k$ lattice points of $P$; a unimodular triangulation implies IDP.
We introduce the following $2$-local substitute.

\begin{definition}[Dyadic decomposition property]\label{definition:dyadic_decomposition}
    A lattice polytope $P\subseteq\RR^n$ has the \emph{dyadic decomposition property with exponent
    $m$}, where $m\in\ZZ_{\ge0}$, if for every $k\ge1$ and every $x\in kP\cap\ZZ^n$ there are
    lattice points $p_1,\ldots,p_{2^mk}\in P\cap\ZZ^n$, not necessarily distinct, such that
    \[
        2^m x=p_1+\cdots+p_{2^mk}.
    \]
    We say that $P$ has the \emph{dyadic decomposition property} if it has this property with some
    exponent~$m$.
\end{definition}

The property with exponent $0$ is precisely IDP, and if an exponent $m$ works, then so does every
larger exponent. Morales has shown that exponent $0$ need not work for a delta-matroid
polytope~\cite{Morales_2026}. The dyadic decomposition property may be viewed as a natural primal
analogue of total dual dyadicness~\cite{Abdi_Cornuejols_Guenin_Tuncel_2024}.

\begin{proposition}[Dyadic triangulations imply dyadic decomposition]
    \label{proposition:dyadic_triangulation_decomposition}
    Let $P\subseteq\RR^n$ be a lattice polytope admitting a dyadic triangulation $\mathcal T$
    with vertices in $P\cap\ZZ^n$, and let
    \[
        2^m\coloneqq\max_{T\in\mathcal T_{\max}}\nvol(T).
    \]
    Then $P$ has the dyadic decomposition property with exponent $m$.
\end{proposition}

\begin{proof}
    Let $k\ge1$ and $x\in kP\cap\ZZ^n$. Since $x/k\in P$, it lies in some maximal cell
    $T=\conv(v_0,\ldots,v_d)$ of $\mathcal T$, say
    $x/k=\sum_{i=0}^d\lambda_iv_i$ with $\lambda_i\ge0$ and $\sum_i\lambda_i=1$. Write
    $\nvol(T)=2^{a}$ with $a\le m$ and $\Lambda=\dir(T)\cap\ZZ^n$. The vector $x-kv_0$ is an integer
    vector of $\dir(T)$, hence lies in $\Lambda$, and by \Cref{lemma:dyadic_bridge} the edge lattice
    $\ZZ\langle v_i-v_0\rangle_{i\ge1}$ has index $2^{a}$ in $\Lambda$. Therefore
    $2^{a}(x-kv_0)\in\ZZ\langle v_i-v_0\rangle_{i\ge1}$. Since the edge vectors are linearly
    independent, uniqueness of their real coordinates gives $\mu_i\coloneqq 2^{a}k\lambda_i\in\ZZ$
    for $1\le i\le d$; and $\mu_0\coloneqq 2^{a}k-\sum_{i\ge1}\mu_i=2^ak\lambda_0$ is then an
    integer as well. All the $\mu_i$ are nonnegative, they sum to $2^{a}k$, and
    \[
        2^{a}x\;=\;2^{a}k\sum_{i}\lambda_iv_i\;=\;\sum_i\mu_iv_i .
    \]
    Multiplying by $2^{m-a}$ gives the statement, the $p_j$ being the $v_i$ repeated
    $2^{m-a}\mu_i$ times.
\end{proof}

\begin{corollary}[Uniform dyadic decomposition in type~$B$]\label{corollary:dyadic_IDP}
    For each $n$ there is an exponent $m_n$ such that every integral type~$B_n$ generalized
    permutohedron $P\subseteq\RR^n$ has the dyadic decomposition property with exponent $m_n$.
    Equivalently, for every $k\ge1$ and every $x\in kP\cap\ZZ^n$, there are lattice points
    $p_1,\ldots,p_{2^{m_n}k}\in P\cap\ZZ^n$ such that
    \[
        2^{m_n}x=p_1+\cdots+p_{2^{m_n}k}.
    \]
    The same conclusion may be phrased in type~$C_n$ terminology only with the ambient lattice
    still equal to $\ZZ^n$.
\end{corollary}

\begin{proof}
    Fix the $B$-generic functionals and one
    $0<\varepsilon<\varepsilon_n$ from \Cref{theorem:main}. There are only finitely many
    delta-matroids on $[n]$, so among all maximal simplices in the triangulations induced by
    $\varphi$ on their polytopes there is a largest normalized volume, say $2^{m_n}$. In the dicing
    construction of \Cref{corollary:typeB_GP}, the triangulation on every unit-box cell, after
    integer translation, is one of these triangulations. Consequently every maximal simplex of
    the resulting triangulation has normalized volume at most $2^{m_n}$, independently of $P$.
    Applying \Cref{proposition:dyadic_triangulation_decomposition} and, when necessary, increasing
    the resulting exponent to $m_n$ proves the statement.
\end{proof}

\begin{remark}[Standard enumerative consequences]\label{remark:enumerative}
Let $P$ have dimension $d$ and let $\mathcal T$ be one of the triangulations above. For
$T\in\mathcal T_{\max}$, write $\nvol(T)=2^{k_T}$. Additivity of normalized volume gives
\[
    \nvol(P)=\sum_{T\in\mathcal T_{\max}}\nvol(T)
             =\sum_{T\in\mathcal T_{\max}}2^{k_T}.
\]
There is also a standard cancellation-free refinement for Ehrhart series. Choose a generic point
in the relative interior of $P$ and make every maximal simplex half-open by removing the facets
visible from that point. These half-open simplices partition $P$~\cite{Koeppe_Verdoolaege_2008}.
If $H_T$ is the half-open simplex associated with $T$, their fundamental parallelepipeds give
\[
    h^*_P(z)=\sum_{T\in\mathcal T_{\max}}h^*_{H_T}(z),
    \qquad h^*_{H_T}(1)=\nvol(T)=2^{k_T},
\]
where every $h^*_{H_T}$ has nonnegative integer coefficients. These are general consequences of a
lattice triangulation; dyadicness supplies the power-of-two value of each summand at $z=1$.
\end{remark}

\begin{remark}[Alcoved polytopes and the dual viewpoint]\label{remark:AB_sides}
    Lattice polytopes with type~$B$ root facet normals are the type~$B$ \emph{alcoved polytopes} of
    Lam and Postnikov~\cite{Lam_Postnikov_2007,Lam_Postnikov_2018}; they admit regular unimodular
    triangulations~\cite[Proposition~3.4]{Haase_Paffenholz_Piechnik_Santos_2021}. This facet-normal
    result is distinct from our edge-direction theorem; compare \Cref{proposition:facet_normals}.
\end{remark}

\end{document}